\newfont{\bcb}{msbm10}
\newfont{\matb}{cmbx10}
\newfont{\got}{eufm10}

\documentclass[12pt]{amsart}
\usepackage{amsmath, amsthm, amscd, amsfonts, amssymb, latexsym, graphicx, color}
\usepackage[bookmarksnumbered, colorlinks, plainpages, hypertex]{hyperref}

\usepackage[cp1250]{inputenc}

\usepackage{amsmath,amsthm}
\usepackage{amssymb,latexsym}
\usepackage{enumerate}

\newtheorem{theorem}{Theorem}[section]
\newtheorem{lemma}[theorem]{Lemma}
\newtheorem{proposition}[theorem]{Proposition}
\newtheorem{corollary}[theorem]{Corollary}
\theoremstyle{definition}

\theoremstyle{remark}
\newtheorem{remark}[theorem]{Remark}
\numberwithin{equation}{section}

\begin{document}

\title[Definable retractions over Henselian fields]
      {Definable retractions and \\
      a non-Archimedean Tietze--Urysohn \\
      theorem over Henselian valued fields}

\author[Krzysztof Jan Nowak]{Krzysztof Jan Nowak}




\date{}

\begin{abstract}
We prove the existence of definable retractions onto arbitrary
closed subsets of $K^{n}$ definable over Henselian valued fields
$K$. Hence directly follows non-Archimedian analogues of the
Tietze--Urysohn and Dugundji theorems on extending continuous
definable functions. The main ingredients of the proof are a
description of definable sets due to van den Dries, resolution of
singularities and our closedness theorem.
\end{abstract}

\maketitle

\section{Main result}

Fix a Henselian, non-trivially valued field $K$ treated in the
language $\mathcal{L}$ of Denef--Pas. The ground field $K$ is
assumed to be of equicharacteristic zero, not necessarily
algebraically closed and with valuation of arbitrary rank. Let
$v$, $\Gamma$, and $K^{\circ}$ denote the valuation, its value
group and valuation ring, respectively. We consider continuity
with respect to the $K$-topology on $K^{n}$, i.e.\ the topology
induced by the valuation $v$. The word ''definable'' means
''definable with parameters''.

\vspace{1ex}

The main purpose of this paper is the following theorem on the
existence of $\mathcal{L}$-definable retractions onto arbitrary
closed $\mathcal{L}$-definable subsets of the affine space and its
applications to extending continuous $\mathcal{L}$-definable
functions.

\begin{theorem}\label{main}
For each closed $\mathcal{L}$-definable subset $A$ of $K^{n}$,
there exists an $\mathcal{L}$-definable retraction $K^{n} \to A$.
\end{theorem}

The proof of Theorem~\ref{main} will be given in Section~3.
Section~2 deals with $\mathcal{L}$-definable retractions onto
simple normal crossing divisors, to which the general case comes
down via resolution of singularities and our closedness theorem
(see~\cite{Now-1,Now-2} and~\cite{Now-3} for the non-Archimedean
analytic version).

\vspace{1ex}

Section~5 gives two direct applications, namely non-Archimedean
analogues of the Tietze--Urysohn and Dugundji theorems on
extending continuous definable functions. Note finally that what
makes to a great extent the character of non-Archimedean extension
problems in the definable case different from the one in the
purely topological case is i.a.\ lack of definable Skolem
functions. On the other hand, definability in a suitable language
often makes the subject under study tamer and enables application
of new tools and techniques.

\section{Retractions onto snc divisors and snc varieties}

Let
$$ B_{n}(r) := \{ a \in K^{n}: v(a) := \min \{ v(a_{1}),\ldots,v(a_{n}) \} > r \} $$
be the $n$-ball of center zero and radius $r$. We begin by stating
the following theorem on continuity of intersections, which is a
direct consequence of the closedness theorem.

\begin{proposition}\label{intersection}
Let $F_{1},\dots,F_{m}$ be closed $\mathcal{L}$-definable subsets
of $(K^{\circ})^{n}$. Then for every $r \in \Gamma$ there is a
$\rho \in \Gamma$ such that
$$ (F_{1} + B_{n}(\rho)) \cap \ldots \cap (F_{m} + B_{n}(\rho))
   \subset (F_{1} \cap \ldots \cap F_{m}) + B_{n}(r). $$
\end{proposition}

\begin{proof}
For the contrary, suppose that there is an $r \in \Gamma$ such
that for every $\rho \in \Gamma$ the set
$$ (F_{1} + B_{n}(\rho)) \cap \ldots \cap (F_{m} + B_{n}(\rho)) \setminus
   ((F_{1} \cap \ldots \cap F_{m}) + B_{n}(r)) \neq \emptyset $$
is non-empty. Put
$$ E := \{ (t,x) \in K^{\circ} \times (K^{\circ})^{n}: $$
$$ x \in
   (F_{1} + B_{n}(v(t))) \cap \ldots \cap (F_{m} + B_{n}(v(t))) \setminus
   ((F_{1} \cap \ldots \cap F_{m}) + B_{n}(r)), $$
and let $\pi: K^{\circ} \times (K^{\circ})^{n} \to K^{\circ}$ be
the projection onto the first factor. Then $0$ is an accumulation
point of $\pi(E)$. It follows from the closedness theorem that
there is an accumulation point $(0,a)$ of the set $E$. Then for
any $\rho \in \Gamma$, $ \rho \geq r$, there are elements
$$ a_{1,\rho} \in F_{1}, \ldots, a_{m,\rho} \in F_{m} \ \
   \text{and} \ \ b_{1,\rho}, \ldots, b_{m,\rho} \in B_{n}(\rho) $$
such that
$$ a_{1,\rho} + b_{1,\rho} = \ldots = a_{m,\rho} + b_{m,\rho} \in
   (a + B_{n}(\rho)) \setminus ((F_{1} \cap \ldots \cap F_{m}) + B_{n}(r)). $$
Then
$$ a_{i,\rho} \in (a + B_{n}(\rho)) \setminus ((F_{1} \cap \ldots \cap F_{m}) +
   B_{n}(r)) \ \ \ \text{for} \ \ i=1,\ldots,m. $$
Hence
$$ a \not \in (F_{1} \cap \ldots \cap F_{m}) + B_{n}(r) $$
and $a$ lies in the closure of every set $F_{i}$ whence
$$ a \in F_{1} \cap \ldots \cap F_{m}. $$
But this is impossible, which finishes the proof.
\end{proof}

We obtain several direct consequences.

\begin{corollary}
Under the above assumption, if $F_{1} \cap \ldots \cap F_{m} =
\emptyset$, then there is a $\rho \in \Gamma$ such that
$$ F_{1} + B_{n}(\rho)) \cap \ldots \cap (F_{m} + B_{n}(\rho) =
   \emptyset. $$
\hspace*{\fill} $\Box$
\end{corollary}

\begin{corollary}
Let $U_{1},\ldots,U_{m}$ be an open $\mathcal{L}$-definable
covering of $(K^{\circ})^{n}$. Then there exists a finite clopen
$\mathcal{L}$-definable partitioning $\{ \Omega_{l}: l=1,\ldots,m
\}$ of $(K^{\circ})^{n}$ such that $\Omega_{l} \subset U_{l}$ for
$l=1,\ldots,m$.
\end{corollary}

\begin{proof}
Apply the above corollary to the complements $F_{l} :=
(K^{\circ})^{n} \setminus U_{l}$. Then the sets $F_{l} +
B_{n}(\rho)$ are clopen by the closedness theorem, and their
complements in $(K^{\circ})^{n}$ are thus clopen
$\mathcal{L}$-definable covering of $(K^{\circ})^{n}$. Hence the
conclusion follows immediately.
\end{proof}

\begin{corollary}\label{partition}
Let $X$ be a closed $\mathcal{L}$-definable subset of
$(K^{\circ})^{n}$ and $\{ U_{l}: l=1,\ldots,m \}$ a finite open
$\mathcal{L}$-definable covering of $X$. Then there exists a
finite clopen $\mathcal{L}$-definable partitioning $\{ \Omega_{l}:
l=1,\ldots,m \}$ of $X$ such that $\Omega_{l} \subset U_{l}$ for
$l=1,\ldots,m$.  \hspace*{\fill} $\Box$
\end{corollary}

\begin{corollary}\label{cor-part}
Let $X$ be a closed $\mathcal{L}$-definable subset of
$\mathbb{P}^{n}(K)$ and $\{ U_{l}: l=1,\ldots,m \}$ a finite open
$\mathcal{L}$-definable covering of $X$. Then there exists a
finite clopen partitioning $\{ \Omega_{l}: l=1,\ldots,m \}$ of $X$
such that $\Omega_{l} \subset U_{l}$ for $l=1,\ldots,m$.
\hspace*{\fill} $\Box$
\end{corollary}

Let $M$ be a non-singular $K$-variety. Recall that a divisor $H$
on $M$ is a {\em simple normal crossing} divisor at a point $a \in
M$ if it is given, in suitable local coordinates $x_{1},
\ldots,x_{n}$ at $a$, by an equation
$$ x_{i_{1}} \cdot \ldots \cdot x_{i_{k}} = 0, \ \
   1 \leq i_{1} < \ldots < i_{k} \leq n, $$
in a Zariski open neighborhood of $a$. If this holds at every
point $a \in M$, we say that $H$ is a {\em simple normal crossing}
divisor (abbreviated to {\em snc} divisor).

We shall need the following version of resolution of singularities
(see e.g.~\cite[Chap.~III]{Kol}.

\begin{theorem}\label{SNC}
Given an algebraic subvariety $V$ of the projective space
$K\mathbb{P}^{n}$, there exists a finite composite $\sigma: M \to
K\mathbb{P}^{n}$ of blow-ups along smooth centers such that the
pre-image $\sigma^{-1}(V)$ is a s.n.c.\ divisor and which is an
isomorphism over the complement of $V$. This means that there is a
Zariski open covering $\{ U_{l}: l=1,\ldots,m \}$ of $M$ such that
the trace $\sigma^{-1}(V) \cap U_{l}$ is a s.n.c.\ divisor in
suitable local coordinates $x_{1}, \ldots,x_{n}$ on $U_{l}$ for
each $l=1,\ldots,m$. \hspace*{\fill}$\Box$
\end{theorem}

Denote by $M(K)$ the set of all $K$-rational points of a
$K$-variety $M$. Theorem~\ref{SNC} and Corollary~\ref{cor-part}
yield immediately

\begin{corollary}\label{cor-SNC}
Given an algebraic subset $V$ of the projective space
$\mathbb{P}^{n}(K)$, there exist a finite composite $\sigma: M \to
K\mathbb{P}^{n}$ of blow-ups along smooth centers and a clopen
$\mathcal{L}$-definable partitioning $\{ U_{l}: l=1,\ldots,m \}$
of $M(K)$ such that the trace $\sigma^{-1}(V) \cap U_{l}$ is a
s.n.c.\ divisor in suitable local coordinates $x_{1},
\ldots,x_{n}$ on $U_{l}$ for each $l=1,\ldots,m$.
\hspace*{\fill}$\Box$
\end{corollary}

The above manifold $M$ can be of course regarded as a non-singular
subvariety of a projective space $K\mathbb{P}^{N}$. In this paper
we are interested just in the $K$-rational points of
$K$-varieties. By abuse of notation, we shall use the same letter
$M$ to designate the $K$-rational points of a given $K$-variety
when no confusion can arise.

\vspace{1ex}

Fix a clopen $\mathcal{L}$-definable chart $U = U_{l}$ on $M$ with
local coordinates $x_{1},\ldots,x_{n}$ considered in
Corollary~\ref{cor-SNC}. We first prove the following

\begin{proposition}\label{retract}
Every smooth divisor $H_{i} := \{ x \in U: x_{i} =0 \}$ is a
retract of a clopen neighborhood $\/\Omega_{i}\/$ of $\/H_{i}$ in
$U$, with $\mathcal{L}$-definable retraction
$$ \omega_{i}: \Omega_{i} \to H_{i}, \ \ (x_{1},\ldots,x_{n}) \mapsto
   (x_{1},\ldots,x_{i-1},0,x_{i+1},\ldots,x_{n}), $$
induced by the local coordinates on $U$.
\end{proposition}

\begin{proof}
Note that a polynomial map $f:K^{N} \to K^{N}$, $f(a)=b$, with
coefficients in $K^{\circ}$ and non-zero Jacobian $e(a) \neq 0$ at
$a$, is an open embedding of the $N$-ball $B_{N}(a,v(e(a)))$ onto
the $N$-ball $B_{n}(b,2 \cdot v(e(a)))$
(cf.~\cite[Proposition~2.4]{Now-2}); here
$$ B_{N}(a,r) := \{ x \in K^{N}: v(x-a) > r \} $$
is the ball with center $a$ and radius $r$.

\vspace{1ex}

Similarly, for an implicit function $y = f(x)$ given by a finite
number of polynomial equations with coefficients in $K^{\circ}$
and for each point $(a,b)$ of its graph, $f$ is uniquely
determined in a polydisk
$$ \{ (x,y): \ v(x) > 2 \cdot v(e(a,b)), \ v(y) > v(e(a,b)) \}, $$
where $e(a,b) \neq 0$ is the suitable minor of the Jacobian matrix
of those equations (cf.~\cite[Proposition~2.5]{Now-2}). Apply
these facts to the suitable equations of the non-singular
subvariety $M$ of $\mathbb{P}^{N}(K)$ and to the coordinate map
$\phi = (x_{1},\ldots,x_{n})$ on $U$. Since the suitable minors do
not vanish on $U$, it follows from the closedness theorem that the
valuation of those minors are uniformly bounded from above on $U$.
Hence there is an $r \in \Gamma$ such that for each point $a \in
U$ the coordinate map $\phi$ is injective on $U \cap B_{N}(a,r)$;
here balls are with respect to the ambient projective space
$\mathbb{P}^{N}(K)$. Then the following formulas in the local
coordinates on $U$:
$$ \omega_{i}(x_{1},\ldots,x_{n}) =
   (x_{1},\ldots,x_{i-1},0,x_{i+1},\ldots,x_{n}), \ \ i=1,\ldots,n, $$
determine well defined retractions $\omega_{i}$ of the clopen (by
virtue of the closedness theorem) $r$-hulls
$$ \Omega_{i} := (H_{i} + B_{N}(0,r)) \cap U $$
of the divisors $H_{i}$ in $U$, we are looking for.
\end{proof}

By a coordinate submanifold of $U$ we mean one of the form
$$ C_{\alpha} := \{ x \in U: x_{\alpha_{1}} = \ldots = x_{\alpha_{p}} = 0
   \}, $$
where $\alpha = (\alpha_{1}, \ldots, \alpha_{p})$ and $1 \leq
\alpha_{1} < \ldots < \alpha_{p} \leq n$. An snc subvariety is a
finite union of coordinate submanifolds of $U$.

It is not difficult to deduce from the closedness theorem that
there is a $\rho \in \Gamma$ such that
$$ U_{\alpha} := \{ x \in U: v(x_{\alpha_{1}}), \ldots, v(x_{\alpha_{p}})
   > \rho \} \subset \Omega_{\alpha} := (C_{\alpha} + B_{N}(0,r)) \cap U $$
for all $\alpha = (\alpha_{1}, \ldots, \alpha_{p})$ as above. The
proof of Proposition~\ref{retract} yields immediately the
following

\begin{corollary}\label{ret-1}
The formulas in the local coordinates on $U$:
$$ \omega_{i}(x_{1},\ldots,x_{n}) =
   (x_{1},\ldots,x_{i-1},0,x_{i+1},\ldots,x_{n}), \ \ i=1,\ldots,n, $$
determine well defined retractions $\omega_{i}$ of the clopen
neighborhoods $U_{i}$ onto the divisors $H_{i}$, respectively.
\hspace*{\fill} $\Box$
\end{corollary}

Similarly, we obtain

\begin{corollary}\label{ret-2}
For every $\alpha$ as before, the formula in local coordinates on
$U$:
$$ \omega_{\alpha}(x) = (x_{1},\ldots,x_{\alpha_{1}-1},0,x_{\alpha_{1}+1},\ldots,
   x_{\alpha_{p}-1},0,x_{\alpha_{p}+1}, \ldots, x_{n}). $$
determines a well defined retraction $\omega_{\alpha}$ of the
clopen neighborhood $U_{\alpha}$ onto the coordinate submanifold
$C_{\alpha}$.  \hspace*{\fill} $\Box$
\end{corollary}

The above results can be strengthened as follows.

\begin{proposition}\label{ret-3}
Every snc subvariety $C$ of $U$, which is a finite union of
coordinate submanifolds $C_{\alpha}, C_{\beta}, \dots$, is an
$\mathcal{L}$-definable retract of the union $U_{C}$ of the clopen
neighborhoods $U_{\alpha}, U_{\beta}, \dots$.
\end{proposition}

\begin{proof}
The proof is by induction on the number of coordinate
submanifolds. Here we consider only the case of two coordinate
submanifolds, say $C_{\alpha}$ and $C_{\beta}$ with $\alpha =
(\alpha_{1}, \ldots, \alpha_{p})$ and $\beta = (\beta_{1}, \ldots,
\beta_{q})$, leaving the general induction step for the reader.

Clearly, the restrictions of $\omega_{\alpha}$ and
$\omega_{\beta}$ to the clopen subsets $U_{\alpha} \setminus
U_{\beta}$ and $U_{\beta} \setminus U_{\alpha}$ are retractions
onto $C_{\alpha} \setminus C_{\beta}$ and $C_{\beta} \setminus
C_{\alpha}$, respectively. Thus it remains to find an
$\mathcal{L}$-definable refraction $\omega$ of $U_{\beta} \cap
U_{\alpha}$ onto $C_{\beta} \cap C_{\alpha}$. Obviously,
$C_{\alpha} \cap C_{\beta} = C_{\gamma}$ for a unique $\gamma$,
and we have $U_{\alpha} \cap U_{\beta} = U_{\gamma}$. Put
$$ d_{\alpha}(x) := \min \{
   v(x_{\alpha_{1}}),\ldots,v(x_{\alpha_{p}}) \}, \ \
   d_{\beta}(x) := \min \{
   v(x_{\beta_{1}}),\ldots,v(x_{\beta_{q}}) \}, $$
$$ D_{\alpha} := \{ x \in U_{\gamma}: d_{\alpha}(x) > d_{\beta}(x) \}, \
   \ D_{\beta} := \{ x \in U_{\gamma}: d_{\beta}(x) > d_{\alpha}(x) \}, $$
$$ \text{and} \ \  D_{\gamma} := \{ x \in U_{\gamma}: d_{\alpha}(x) =
   d_{\beta}(x) \}. $$
It is easy to check that the formula
$$ \omega_{C}(x) = \left\{ \begin{array}{cl}
                        \omega_{\alpha}(x) & \mbox{ if }  x \in D_{\alpha}, \\
                        \omega_{\beta}(x) & \mbox{ if }  x \in D_{\beta}, \\
                        \omega_{\gamma}(x) &  \mbox{ if } x \in D_{\gamma}.
                       \end{array}
               \right.
$$
defines a retraction $\omega_{C}:U \to C$, we are looking for.
\end{proof}

Since $U$ is a clopen $\mathcal{L}$-definable subset of $U$, we
immediately obtain

\begin{corollary}\label{ret-4}
Every snc subvariety $C$ of $U$ is an $\mathcal{L}$-definable
refract of $U$ and of $M$.  \hspace*{\fill} $\Box$
\end{corollary}

\begin{remark}\label{affine}
Let $V_{\infty} \subset \mathbb{P}^{n}(K)$ be the hyperplane at
infinity. In the foregoing reasonings, take into account both the
subset $V$ and $V_{\infty}$ with respect to simultaneous
transformation to s.n.c.\ divisors. Then it is not difficult to
check that the above results are valid mutatis mutandi in the case
where we consider the charts over the affine space $K^{n}$:
$$ U_{l}^{0} := U_{l} \setminus \sigma^{-1}(V_{\infty}), \ \ j=1,\ldots,s, $$
on $M^{0} := M \setminus  \sigma^{-1}(V_{\infty})$. Applying this
to Corollary~\ref{ret-4}, we see that $C^{0} := C \setminus
\sigma^{-1}(V_{\infty})$ is an $\mathcal{L}$-definable refract of
$U_{0}$ and of $M_{0}$

\end{remark}

\section{Proof of Theorem~\ref{main}}

We first prove the descent property:

\begin{lemma}\label{descent}
Let $\sigma: M \to N$ be a continuous $\mathcal{L}$-definable
surjective map which is $\mathcal{L}$-definably closed, and $A$ be
a closed $\mathcal{L}$-definable subset of $N$ such that $\sigma$
is bijective over $N \setminus A$. Then if $\varrho: M \to
\sigma^{-1}(A)$ is an $\mathcal{L}$-definable retraction, so is
the map $\omega: N \to A$ defined by the formula
$$ \omega(a) = \left\{ \begin{array}{cl}
                        \sigma \left( \varrho \left( \sigma^{-1}(a) \right) \right) & \mbox{ if } \ a \in N \setminus A, \\
                        a & \mbox{ if } \ a \in A.
                        \end{array}
               \right.
$$
\end{lemma}

\begin{proof}
We must show that the map $\omega$ is continuous. Take any closed
$\mathcal{L}$-definable subset $E \subset N$. Then
$$ \sigma^{-1} \left( \omega^{-1} (E) \right) = \varrho^{-1}
   \left( \sigma^{-1} (E) \right) $$
is a closed subset of $M$. Hence $\omega^{-1}(E) = \sigma \left(
\sigma^{-1} \left( \omega^{-1} (E) \right) \right)$ is a closed
subset of $N$, which is the required result.
\end{proof}

In view of Remark~\ref{affine}, Corollary~\ref{ret-4} along with
Lemma~\ref{descent} yield the following special case of
Theorem~\ref{main}.

\begin{proposition}\label{special}
For each Zariski closed subset $Y$ of $K^{n}$, there exists an
$\mathcal{L}$-definable retraction $\varrho: K^{n} \to Y$.

\end{proposition}

We now state a description of $\mathcal{L}$-definable sets due to
van den Dries~\cite{Dries}; see~\cite[Section~9]{Now-2} for the
adaptation to the language of Denef-Pas considered here.

\begin{proposition}\label{description}
Every $\mathcal{L}$-definable subset $A$ of $K^{n}$ is a finite
union of intersections of Zariski closed with special open subsets
of $K^{n}$. A fortiori, $A$ is of the form
\begin{equation}\label{present}
  A = (V_{1} \cap G_{1}) \cup \ldots \cup
  (V_{s} \cap G_{s}),
\end{equation}
where $G_{j}$ are clopen ${\mathcal{L}}$-definable subsets of
$K^{n} \setminus W_{j}$, and $V_{j}$ and $W_{j}$ are Zariski
closed subsets of $K^{n}$, $j=1,\ldots,s$, which may occur with
repetition. We may, of course, assume that the sets $V_{j}$ are
irreducible.  \hspace*{\fill}$\Box$
\end{proposition}

In order to proceed with noetherian induction, we shall rephrase
Theorem~\ref{main} as follows.

\begin{theorem}\label{induct}
Let $X$ be a Zariski closed subset of $K^{n}$. For each closed
$\mathcal{L}$-definable subset $A$ of $X$, there exists an
$\mathcal{L}$-definable retraction $X \to A$.
\end{theorem}

\begin{proof}

Let $Z$ be the Zariski closure of $A$ in $K^{n}$. If $Z
\varsubsetneq X$, then the conclusion of the theorem follows
directly from Proposition~\ref{special} and the induction
hypothesis.

So suppose that $Z=X$ and let $X_{1},\ldots,X_{t}$ be the
irreducible components of $X$. We can, of course, assume that each
irreducible set $V_{j}$ from presentation~\ref{present} is
contained in a component $X_{k}$. Then the set of indices $j$ with
$V_{j}=X_{1}$ and $W_{j} \cap X_{1} \varsubsetneq X_{1}$, say
$1,\ldots,p$, is non-empty. We may assume, after renumbering, that
the indices $j$ with $V_{j} \varsubsetneq X_{1}$ are $p+1,\ldots,
q$. Then
$$ (W_{1} \cap X_{1}) \cup \ldots \cup (W_{p} \cap X_{1})
   \cup V_{p+1} \cup \ldots \cup V_{q} \varsubsetneq X_{1} $$
and
$$ Y := (W_{1} \cap X_{1}) \cup \ldots \cup (W_{p} \cap X_{1})
   \cup V_{p+1} \cup \ldots \cup V_{q} \cup X_{2} \cup \ldots \cup X_{t} \varsubsetneq X. $$

Further, the set
$$ E := \left( (G_{1} \cup \ldots \cup G_{p}) \cap X_{1} \right) \setminus Y $$
is, of course, clopen in $X_{1} \setminus Y = X \setminus Y$.
Clearly, $A = E \cup B$ with $B := A \cap Y$, and $a \in B$ for
every point $a \in Y$ that is an accumulation point of $E$.

By the induction hypothesis, there exists an
$\mathcal{L}$-definable retraction $\psi: Y \to B$. Now, take an
$\mathcal{L}$-definable retraction $\varrho: X \to Y$ from
Proposition~\ref{special}. Then it is easy to check that the map
$\omega: X \to E \cup B$ defined by the formula
$$ \omega(a) = \left\{ \begin{array}{cl}
                        \psi (\varrho (a))  & \mbox{ if } \ a \in X \setminus E, \\
                        a & \mbox{ if } \ a \in E.
                        \end{array}
               \right.
$$
is an $\mathcal{L}$-definable retraction. This completes the proof
of Theorem~\ref{induct} and, a fortiori, of Theorem~\ref{main}.
\end{proof}

\section{A non-Archimedean Dugundji extension theorem}

Theorem~\ref{main} (on the existence of $\mathcal{L}$-definable
retractions onto closed $\mathcal{L}$-definable subsets of the
affine space $K^{n}$) yields a non-Archimedean definable analogue
of the Dugundji on the existence of a linear (and continuous)
extender, stated below. This extension problem, going back to
Dugundji~\cite{Du}, was extensively studied by many specialists
(see e.g.~\cite{D-L-P} for references).

\vspace{1ex}

Let us introduce the following notation. Given a topological space
$X$, denote by $\mathcal{C}^{*}(X,K)$ the $K$-linear space of all
continuous bounded $K$-valued functions on $X$, equipped with the
topology of uniform convergence. This topology is metrizable by
the supremum norm whenever the ground field $K$ is with absolute
value or, equivalently, is rank one valued. For a closed
$\mathcal{L}$-definable subset $A$ of $K^{n}$, denote by
$\mathcal{C}^{\mathcal{L}}(A,K)$ the $K$-linear space of all
continuous $\mathcal{L}$-definable $K$-valued functions on $A$.


\begin{theorem}\label{Dug}
For each closed $\mathcal{L}$-definable subset $A$ of $K^{n}$,
there exists a linear extender
$$ T : \mathcal{C}^{\mathcal{L}}(A,K) \to \mathcal{C}^{\mathcal{L}}(K^{n},K) $$
that is continuous in the topology of uniform convergence.
\end{theorem}

\begin{proof}
Indeed, an $\mathcal{L}$-definable retraction $\omega: K^{n} \to
A$ exists by virtue of Theorem~\ref{main}. The operator $T$
defined by putting $T(f) := f \circ \omega$ is an extender we are
looking for.
\end{proof}

For the sake of comparison, two non-Archimedean, purely
topological versions from the papers~\cite[Cor.~2.5.23]{PG-Sch},
\cite[Theorem~5.24]{Ro} and from~\cite[Theorem~2]{K-K-S},
respectively, are recalled below. Note that the real topological
counterpart of the first one fails unless the space $X$ is
metrizable (cf.~\cite{Du,Ar,Mi}) or, more generally, stratifiable
(cf.~\cite{Bo}).

\vspace{2ex}

\begin{em}
1) Let $K$ be a locally compact field with non-Archimedean
absolute value and $A$ be a closed subspace of an ultranormal
space X. Then there exists an isometric (in the supremum norm)
linear extender

$$ T : \mathcal{C}^{*}(A,K) \to \mathcal{C}^{*}(X,K). $$
such that
$$ \sup\, \{ |Tf(x)| : x \in X \} = \sup\, \{ |Tf(x)| : x \in A \} $$
for all $f \in \mathcal{C}^{*}(X,K)$.
\end{em}

\vspace{2ex}

\begin{em}
2) Let $K$ be a discretely valued field and $A$ be a closed
subspace of an ultranormal space $X$. Then there exists an
isometric linear extender
$$ T : \mathcal{C}^{*}(A,K) \to \mathcal{C}^{*}(X,K) $$
whenever at least one of the following conditions holds:

i) $X$ is collectionwise normal; \

ii) $A$ is Lindel\"{o}f; \

iii) $K$ is separable.
\end{em}

\vspace{2ex}

Theorem~\ref{Dug} a fortiori yields the following non-Archimedean
version of the Tietze--Urysohn extension theorem.

\vspace{1ex}

\begin{theorem}~\label{TU}
Every continuous $\mathcal{L}$-definable function $f: A \to K$ on
a closed subset $A$ of $K^{n}$ has a continuous
$\mathcal{L}$-definable extension $F$ to $K^{n}$. \hspace*{\fill}
$\Box$
\end{theorem}

The classical Tietze--Urysohn extension theorem says that every
continuous (and bounded) real valued map on a closed subset of a
normal space $X$ can be extended to a continuous (and bounded)
function on $X$. Afterwards the problem of extending maps into
metric spaces or locally convex linear spaces was investigated by
several mathematicians, i.al.\ by Borsuk~\cite{Bor},
Hausdorff~\cite{Hau}, Dugundji~\cite{Du}, Arens~\cite{Ar} or
Michael~\cite{Mi}. Eventually, Ellis~\cite{El-2} established some
analogues of their results, concerning the extension of continuous
maps defined on closed subsets of zero-dimensional spaces with
values in various types of metric spaces. They apply, in
particular, to continuous functions from ultranormal spaces into a
complete separable field with non-Archimedean absolute value and
to continuous functions from ultraparacompact spaces into an
arbitrary complete field with non-Archimedean absolute value.
Hence follows his analogue of the Tietze--Urysohn theorem
from~\cite{El-1} on extending continuous functions from
ultranormal spaces into a locally compact field with
non-Archimedean absolute value.

\vspace{1ex}

Note that ultranormal spaces are precisely those of great
inductive dimension zero (cf.~\cite[Chap.~7]{En}) and that the
class of ultraparacompact spaces coincides with that of
ultranormal and paracompact spaces. Finally, let us mention that
the projective spaces $\mathbb{P}^{n}(K)$, and even the affine
spaces $K^{n}$, are definably ultranormal by virtue of the
closedness theorem; the latter result, however, requires more
careful analysis.

\vspace{1ex}

We conclude the paper with the following comment.

\begin{remark}
The main results of this paper hold also over complete fields with separated power series and even over Henselian fields with analytic structure, as established in our recent papers~\cite{Now-4,Now-5}. Our interest to problems of non-Archimedean geometry was inspired by our joint paper~\cite{K-N}. Let us finally mention that the theory of Henselian fields with analytic structure, which unifies many earlier approaches within non-Archimedean analytic geometry, was developed in the papers~\cite{L-R,C-Lip-R,C-Lip-0,C-Lip}.
\end{remark}


\vspace{3ex}

\begin{small}
Institute of Mathematics

Faculty of Mathematics and Computer Science

Jagiellonian University


ul.~Profesora S.\ \L{}ojasiewicza 6

30-348 Krak\'{o}w, Poland

{\em E-mail address: nowak@im.uj.edu.pl}
\end{small}


\begin{thebibliography}{99}



\bibitem{Ar}
R.~Arens, {\em Extension of functions on fully normal spaces},
Pac.\ J.\ Math.\ {\bf 2} (1952), 11–-22.

\bibitem{Bo}
C.J.R.~Borges, {\em On stratifiable spaces}, Pac.\ J.\ Math.\ {\bf
17} (1966), 1–16

\bibitem{Bor} K.~Borsuk, {\em \"{U}ber Isomorphie der
Funktionalr\"{a}ume}, Bull.\ Int.\ Acad.\ Polon.\ Ser.\ A {\bf
1/3} (1933), 1--10.


\bibitem{C-Lip-R}
R.~Cluckers, L.~Lipshitz, Z.~Robinson, \emph{Analytic cell
decomposition and analytic motivic integration}, Ann.\ Sci.\ École
Norm.\ Sup.\ (4) {\bf 39} (2006), 535--568.

\bibitem{C-Lip-0}
R.~Cluckers, L.~Lipshitz, {\em Fields with analytic structure\/},
J.\ Eur.\ Math.\ Soc.\ {\bf 13} (2011), 1147--1223.

\bibitem{C-Lip}
R.~Cluckers, L.~Lipshitz, {\em Strictly convergent analytic
structures\/}, J.\ Eur.\ Math.\ Soc.\ {\bf 19} (2017), 107--149.


\bibitem{D-L-P}
E.K.~van Douwen, D.J.~Lutzer, T.C.~Przymusi\'{n}ski, {\em Some
extensions of the Tietze--Urysohn theorem}, Amer.\ Math.\ Mon.\
{\bf 84} (1977), 435–-441.

\bibitem{Dries}
L.~van den Dries, \emph{Dimension of definable sets, algebraic
boundedness and Henselian fields}, Ann.\ Pure Appl.\ Logic {\bf
45} (1989), 189--209.

\bibitem{Du}
J.~Dugundji, {\em An extension of Tietze’s theorem}, Pac.\ J.\
Math.\ {\bf 1} (1951), 353–-367.

\bibitem{El-1}
R.L.~Ellis, {\em A non-Archimedean analogue of the Tietze-Urysohn
extension theorem}, Indag.\ Math.\ {\bf 29} (1967), 332–-333.

\bibitem{El-2}
R.L.~Ellis, {\em Extending continuous functions on
zero-dimensional spaces}, Math.\ Ann.\ {\bf 186} (1970), 114–-122.

\bibitem{En}
R.~Engelking, {\em General Topology}, Sigma Series in Pure
Mathematics {\bf 6}, Heldermann, Berlin, 1989.



\bibitem{Hau}
F.~Hausdorff, {\em Erweiterung einer stetigen Abbildung},
Fundamenta Math.\ {\bf 30} (1938), 40--47.

\bibitem{K-K-S}
J.~K\c{a}kol, A.~Kubzdela, W.~\'{S}liwa, {\em A non-Archimedean
Dugundji extension theorem}, Czechoslovak Math.\ J.\ {\bf 63}
(138) (2013), 157--164.

\bibitem{Kol}
J.~Koll{\'a}r, Lectures on resolution of singularities, Ann,\
Math.\ Studies, Vol.\ 166, Princeton Univ.\ Press, Princeton, New
Jersey, 2007.

\bibitem{K-N}
J.~Koll{\'a}r, K.~Nowak, {\em Continuous rational functions on
real and $p$-adic varieties\/}, Math. Zeitschrift {\bf 279}
(2015), 85--97.

\bibitem{L-R}
L.~Lipshitz, Z.~Robinson, {\em Uniform properties of rigid
subanalytic sets\/}, Trans.\ Amer.\ Math.\ Soc.\ {357} (11)
(2005), 4349--4377.

\bibitem{Mi}
E.~Michael, {\em Some extension theorem for continuous functions},
Pac.\ J.\ Math.\ {\bf 3} (1953), 789–-806.

\bibitem{Now-1}
K.J.~Nowak, {\em Some results of algebraic geometry over Henselian
rank one valued fields\/},  Sel.\ Math.\ New Ser.\ {\bf 23}
(2017), 455--495.

\bibitem{Now-2}
K.J.~Nowak, {\em A closedness theorem and applications in geometry
of rational points over Henselian valued fields}, arXiv:1706.01774
[math.AG] (2017).

\bibitem{Now-3}
K.J.~Nowak, {\em Some results of geometry over Henselian fields
with analytic structure}, arXiv:1808.02481 [math.AG] (2018).

\bibitem{Now-4}
K.J.~Nowak, {\em Definable retractions over complete fields with
separated power series}, arXiv:1901.00162 [math.AG] (2019).

\bibitem{Now-5}
K.J.~Nowak, {\em Definable transformation to normal crossings over Henselian fields
with separated analytic structure}, arXiv:1903.07142 [math.AG] (2019).


\bibitem{PG-Sch}
C.~Perez-Garcia, W.H.~Schikhof, {\em Localy Convex Spaces Over
Non-Archimedean Valued Fields}, Cambridge Univ.\ Press, Cambridge,
2010.

\bibitem{Ro}
A.C.M.~van~Rooij, {\em Non-Archimedean Functional Analysis},
Monographs and Textbooks in Pure and Appl.\ Math.\ {\bf 51},
Marcel Dekker, New York, 1978.































\end{thebibliography}
\end{document}